\documentclass[11pt]{article}
\usepackage{amssymb,amsfonts,amsmath,latexsym,epsf,tikz,url}

\newtheorem{theorem}{Theorem}[section]

\newtheorem{corollary}[theorem]{Corollary}
\newtheorem{lemma}[theorem]{Lemma}
\newtheorem{definition}[theorem]{Definition} 
\newtheorem{example}[theorem]{Example}
\newtheorem{remark}[theorem]{Remark}

\newcommand{\proof}{\noindent{\bf Proof.\ }}
\newcommand{\qed}{\hfill $\square$\medskip}

\begin{document}

\title{On the  independent domination polynomial of a graph}

\author{Somayeh Jahari$^{}$\footnote{Corresponding author}
\and
Saeid Alikhani 
}

\date{}

\maketitle

\begin{center}
 Department of Mathematics, Yazd University, 89195-741, Yazd, Iran\\

{\tt s.jahari@gmail.com ~~~~ alikhani@yazd.ac.ir}
\end{center}


\begin{abstract}
 An independent dominating set of  the simple graph $G=(V,E)$ is a vertex subset that is both dominating and independent in $G$. The {\it independent domination polynomial} of  a graph $G$ is the polynomial
$D_i(G,x)=\sum_{A} x^{|A|}$,
summed over all independent dominating subsets $A\subseteq V$. A root of $D_i(G,x)$ is called an independence domination root.  We investigate the independent domination polynomials of some generalized compound graphs. As consequences, we construct 
graphs whose independence domination roots are real.  Also, we consider some certain graphs and study the number of their independent dominating sets.
\end{abstract}


\medskip
\noindent{\bf AMS Subj. Class.:} 05C69

\section{Introduction}

All graphs in this paper are simple of finite orders, i.e., graphs are undirected with no loops or
parallel edges and with finite number of vertices. Let $G$ be a graph. 
A non-empty set $S\subseteq V(G)$ is a {\it dominating set} if every vertex in $V(G)\backslash S$ is adjacent to at least one vertex in $S$ and  the minimum cardinality of all dominating sets of $G$ is called  the {\it domination number} of $G$ and is denoted by $\gamma(G)$.
For a detailed treatment of domination theory, the reader is referred to~\cite{domination}. 

The {\it complement} $\bar{G}$ of a graph $G$ is a graph with the same vertex set as $G$ and with the property that two vertices are adjacent in $\bar{G}$ if and only if they are not adjacent in $G$. 
The {\it line graph} $L(G)$ of a graph $G$ has vertex set $V(L(G))=E(G)$, and two vertices of $L(G)$ are adjacent if and only if they are adjacent as edges of $G$. For a set $S \subseteq V$, the subgraph induced by $S$ is denoted by $\langle S\rangle$.
A graph $G$ is called {\it claw-free} if it contains no an induced subgraph isomorphic to the complete bipartite graph $K_{1,3}$.  
Claw-free graphs form are important and very well-studied class of graphs. There are many natural examples
of claw-free graphs such as line graphs, complements of triangle-free graphs, etc. The complete
set of claw-free graphs has been characterized by  Seymour and Chudnovsky in \cite{i,ii,iii}.


 An {\it independent set} in a graph $G$ is a set of pairwise non-adjacent vertices. A maximum independent set in $G$ is a largest independent set and its size is called {\it independence number} of $G$ and is denoted $\alpha(G)$. A graph is said to be {\it well-covered} if all of its maximal independent sets have the same size. In general, a graph $G$ with independence number $\alpha(G)$ is well-covered if and only if $\bar{G}$ is $K_{\alpha(G)+1}$-free and every clique, a set of vertices all pairwise adjacent, of cardinality less than $\alpha(G)$ is contained in a clique of order $\alpha(G)$.  
 An independent dominating set of $G$ is a vertex subset that is both dominating and independent in $G$, or equivalently, is a maximal independent set. The independent domination number of $G$, denoted by $\gamma_i(G)$, is the minimum size of all independent dominating sets of $G$.  The following relationship among the parameters under consideration is well-known \cite{domination}, 
 \[\gamma(G)\leq \gamma_i(G)\leq \alpha(G).\]

Let $d_i(G,k)$ denote the number of independent dominating sets of $G$ with cardinality $k$, i.e.,
\[d_i(G,k)=| \{D\subseteq V(G)~|~|D|=k,~ \langle D\rangle \textit{~is~an~empty~graph~ and } N[D]=V(G) \}|.\]  
The {\it independent domination polynomial}, $D_i(G,x)$ of $G$ is defined as
\[D_i(G,x)=\sum_{ k=\gamma_i(G)}^{\alpha(G)} d_i(G,k) x^{k}.\] 
Thus $D_i(G,x)$ is the generating polynomial for the number of independent dominating sets of  $G$ of each cardinality. 
A root of $D_i(G, x)$ is called an independence  domination root of $G$.

For many graph polynomials, their roots have attracted considerable attention, both for their own sake, as well for what the nature and location of the roots imply. The roots of the chromatic polynomial, independence polynomial, domination polynomial and total domination polynomials have been studied extensively \cite{aaop,euro,an,ber,bre,br,bt,cs,ob}. 
We investigate here independence  domination roots, that is, the roots of independent domination polynomials.
It is easy to see that the independent domination polynomial has no constant term. Consequently, $0$ is a root of every independent domination polynomial (in fact, $0$ is a root whose multiplicity is the independent domination number of the graph). Because the coefficients are positive integers, $(0,\infty)$ is a zero-free interval.  
Note that the independent domination polynomial of well-covered graph is a monomial which its root is only zero.

Let $a_0, a_1, \dots , a_n$ be a sequence of nonnegative numbers. It is unimodal if there is some
$m$, called a mode of the sequence, such that
\[
a_0 \leq a_1 \leq \dots \leq a_{m-1} \leq a_m \geq a_{m+1} \geq \dots \geq a_n.
\]
It is log-concave if $a_k^2\geq a_{k-1}a_{k+1}$ for all $1 \leq k \leq n - 1$. It is symmetric if $a_k = a_{n - k}$ for $0 \leq k \leq n$. A log-concave sequence of positive numbers is unimodal (see, e.g.,  \cite{Brenti}).
 We say that a polynomial $\sum^n_{k=0} a_k x^k$ is unimodal (log-concave, symmetric, respectively)
 if the sequence of its coefficients $a_0, a_1, \dots , a_n$ is unimodal (log-concave, symmetric, respectively). A basic approach to unimodality problems is to use Newton's inequalities: Let $a_0, a_1, \dots , a_n$ be a sequence of nonnegative numbers. Suppose
that the polynomial  $\sum\limits_{k=0}^n a_k x^k$ has only real zeros. Then

\[a_k^2\geq a_{k-1}a_{k+1}(1+\frac{1}{k})(1+\frac{1}{n-k}),~~ k = 1, 2, \cdots, n - 1,\]
and the sequence is therefore log-concave and unimodal \cite{hl}. 
Unimodality problems of graph polynomials have always been of great interest to researchers in graph theory. For example, it is conjectured that the chromatic polynomial  and domintion polynomial of a graph are unimodal \cite{saeid1,read}. Recently, the authors in \cite{uni}  have shown that if $H=K_r$,  then the polynomial $D(G \circ H, x)$ is unimodal for every graph $G$, where $G\circ H$ is the corona product of two graphs $G$ and $H$ defined by Frucht and Harary \cite{Fruc}. There has been an extensive literature in recent years on the unimodality problems of independence polynomials (see \cite{toc,br,yz,bx} for instance).

\medskip

In the next section, we study the location of independence domination root of graphs. In section  3, we investigate the independence polynomials of some generalized compound graphs, and  we  construct graphs whose  independence polynomials are 
unimodal or log-concave, or having only real zeros. In section 4, we consider specific graphs and study their independent dominating sets with cardinality $i$, for $\gamma_i(G)\leq i \leq \alpha(G)$. The independent domination polynomial of some standard graphs are
obtained, some properties of the independent domination polynomial of a graph are established. 


  \section{Location of independence domination roots }
 In this section, we consider the location of roots of independent domination polynomials. Of course, such roots must necessarily be negative as independent domination polynomials have positive coefficients.  

 The integer roots of graph polynomials  have been extensively studied in the literature, see e.g. \cite{aaop,dong}. There  is  a conjecture in \cite{euro} 
  which states that every integer root of the domination polynomial $D(G,x)$ is $-2$ or $0$. 

 Let $G$ be a graph of order $n$. Note that $D_i(G,1)$ equals the number of independent dominating sets of $G$. Also $D_i(G,-1)$ is the difference of the numbers of independent dominating sets of even size and odd size of $G$.   
  The join $ G_1+ G_2$ of two graphs $G_1$ and $G_2$ with disjoint vertex sets $V_1$ and $V_2$ and edge sets $E_1$ and $E_2$ is the graph union $G_1\cup G_2$ together with all the edges joining $V_1$ and $V_2$.
 It is not hard to see that the formula for independent domination polynomial of join of two graphs is obtained as follows. The following result was also proven in \cite{Dod} as Theorem 2.3:

\begin{theorem}{\rm \cite{Dod}}
If  $G_1$ and $G_2$ are  nonempty graphs, then
\begin{eqnarray*}
D_i(G_1 + G_2,x) = D_i(G_1, x) + D_i(G_2, x).
\end{eqnarray*}
\end{theorem}
\begin{proof}
 By definition, for every  independent dominating set of $G_1~(G_2), D\subseteq V(G_1)$\\$~(D\subseteq V(G_2)), D$ is an independent 
 dominating set for $G_1+G_2$. So we have result.\qed
\end{proof}  
  
  \begin{theorem}
  	\begin{enumerate}
  		\item [(i)]   For any integer number $n$, there is a connected graph $G$  such that $D_i(G,-1)=n$.
  		
  		\item [(ii)]  For any negative integer number $n$, there is a connected graph $G$  for which $n$ is an independence  domination roots of $G$.
  	\end{enumerate}
 \end{theorem}
\begin{proof}
	\begin{enumerate}
		\item [(i)] 	
 Let $G_i$ be a non-empty graph with $D_i(G_i,-1)=1$ (say $P_8$) for every $i\in \{1,2,\cdots,n\}$ and $H=G_1+G_2+\cdots +G_n$. Then 
  \[D_i(H,x)=\sum\limits_{i=1}^n D_i(G_i,x),\]
  and so the result is true for positive integer number $n$. Now for negative integer $n$, consider  a non-empty graph $G_i$ with $D_i(G_i,-1)=-1$.
  
  \item[(ii)]  
   For  natural number $n$, consider complete $(n+1)$-partite graph \\ $G_m=K_{m,m-1,\cdots, m-1}$. 
   Thus
   \[D_i(G_m,x)=x^m+\sum\limits_{i=1}^n x^{m-1}=x^m+nx^{m-1}.\]
   So the result is true for negative integer number $n$.\qed 
  \end{enumerate} 
  \end{proof}

To know more about the location of roots of independence domination polynomials  of  graphs, we recall the definition of lexicographic product of two graphs.  
For two graphs $G$ and $H$, let $G[H]$ be the graph with vertex
set $V(G)\times V(H)$ and such that vertex $(a,x)$ is adjacent to vertex $(b,y)$ if and only if
$a$ is adjacent to $b$ (in $G$) or $a=b$ and $x$ is adjacent to $y$ (in $H$). The graph $G[H]$ is the
lexicographic product (or composition) of $G$ and $H$, and can be thought of as the graph arising from $G$ and $H$ by substituting a copy of $H$ for every vertex of $G$.
The following theorem which is similar to independence polynomial of $G[H]$ (see \cite{br}) gives the independent domination polynomial of $G[H]$.

\begin{theorem}
	If  $G$ and $H$ are two  graphs, then the independent domination polynomial of $G[H]$ is 
	\[ 
	D_i(G[H],x)=D_i(G,D_i(H,x)).
	\]
\end{theorem}
\begin{proof}
	By definition, the polynomial $D_i(G,D_i(H,x))$ is given by
	\begin{equation}\label{eq1}
	\sum\limits_{k=0}^{\alpha(G)}d_i(G,k)\Big(\sum\limits_{j=0}^{\alpha(H)}d_i(H,j)x^j\Big)^k.
	\end{equation}
	An independent dominating set in $G[H]$ of cardinality $l$ arises by choosing an independent dominating
	set in $G$ of cardinality $k$, for some $k\in \{0,1,~\ldots, l\},$ and then, within each copy of $H$ in $G[H]$, 
	choosing an independent dominating set in $H$, in such a way that the total number of vertices chosen is $l$. But the number of ways of actually doing this is exactly the coefficient of $x^l$ in (\ref{eq1}),  which completes the proof.\qed
\end{proof}

We define an expansion of a graph $G$ to be a graph formed from $G$ by replacing each vertex by a complete graph; that is, for each vertex $u$ of $G$, we replace $u$ by a new complete graph $K_u$, and add in edges between all vertices in $K_u$ and $K_v$ whenever $uv$
is an edge of $G$. The expansion operation can push the roots of the  independent domination polynomial into the unit disc.

\begin{theorem}
	Every graph $G$ is an induced subgraph of a graph $H$ whose independence  domination roots lie in $|z|\leq 1$.
\end{theorem}
\begin{proof}
	We replace  every vertex of $G$ by the same suitably large complete graph $K_r$. So we have  
	\[D_i(H, x)=D_i(G[K_r], x)=D_ i(G, rx) =\sum\limits_{k\geq 0}r^kd_i(G,k)x^k.\]
	If $r$ is large enough, then the coefficients of $D_i(H,k)$ can be made to be increasing, and by the Enestr\"{o}m-Kakeya Theorem (c.f. \cite{ander}), 
	all the roots will lie in $|z|\leq 1$.\qed
\end{proof}


\section{Independent domination polynomials of compound graphs}

Song et al. in \cite{song} defined an operation of graphs called the compound graph.   
Given two graphs $G$ and $H$, assume that $\mathcal{C} = \{C_1,C_2, \cdots ,C_k\}$ is a clique cover of $G$.
Construct a new graph from $G$, as follows: for each clique $C_i \in \mathcal{C}$, add a
copy of the graph $H$ and join every vertex of $C_i$ to every vertex of $H$. Let $G^{\Delta}(H)$ denote
the new graph. In fact, the compound graph is a generalization of
the corona of $G$ and $H$, if each clique $C_i$ of the clique cover $\mathcal{C}$ is a vertex.

In this section, we consider compound graphs and  formulate the independent domination polynomial for some generalized compound graphs.  To do this we need some preliminaries.

The independence polynomial was introduced in \cite{gut} as a generalization of the matching polynomial:
\[I(G,x)=\sum_{k\geq 0}i(G,k)x^k,\]
where $i(G,k)$ is the number of independent subsets of $V(G)$ with cardinality $k$. 
We recall that $G$ is called to be claw-free if no induced subgraph of it is a claw. Chudnovsky and Seymour  \cite{cs} showed that the independence polynomial of a claw-free graph has only real zeros. Chudnovsky and  Seymour actually proved the next result.
\begin{theorem}\label{thid}{\rm \cite{cs}}
If  $G$ is  a claw-free graph, then $I(G, x)$ has only real zeros.
\end{theorem}

The following theorem gives the independent domination polynomial of $G^{\Delta}(H)$.  
  \begin{theorem}\label{cd}
 For two graphs $G$ and $H$, let $\mathcal{C} = \{C_1,C_2, \cdots, C_q\}$ be  a clique cover of $G$. Then 
\[ 
	D_i(G^{\Delta}(H),x)=D_i^{q}(H,x) I(G,\frac{x}{D_i(H,x)}).
\]
\end{theorem} 
\proof 
Every independent vertex subset of $G$ can be expanded to an independent dominating set in $G^{\Delta}(H)$.  For each $k$, select $k$ independent elements from $V(G^\mathcal{C} \star H)$ in a two-stage process. First, let us choose $m$ independent elements from $V(G)$ and partition $\mathcal{C}$ into two groups: the cliques containing one of the chosen elements, and those do not. And then select the remaining $(k-m)$ independent dominating elements from $V((q - m)H)$. In consequence, we obtain that $d_i(G^\mathcal{C} \star H,k)$ equals to
\[\sum\limits_{m=0}^ki_m\sum\limits_{j_1+\cdots+j_{q-m}=k-m}d_i(H,j_1)\cdots d_i(H,j_{q-m}).\]
Thus we have
\begin{eqnarray*}
D_i(G^{\Delta}(H),x)&=& \sum\limits_{k\geq 0}i_mx^mD_i^{q-m}(H,x)\\
&=&D_i^{q}(H,x) \sum\limits_{m\geq 0}i_mx^mD_i^{-m}(H,x)\\
&=& D_i^{q}(H,x) I(G,\frac{x}{D_i(H,x)}).
\end{eqnarray*}
 Therefore the Theorem follows. \qed

It is not hard to see that $q -\alpha(G)$ is always nonnegative. So in view of Theorem \ref{cd}, the following is immediate.
\begin{corollary}
Given two graphs $G$ and $H$, assume that $\mathcal{C}$ is a clique cover of $G$. If  $|\mathcal{C}| = q$, then $ [D_i(H,x)]^{q-\alpha(G)}$ divides $D_i(G^{\Delta}(H),x)$.
\end{corollary}

In fact, Theorem \ref{cd} is generalization of the following corollary. 

\begin{corollary}\!{\rm\cite{Dod}}\!
For any graph $G$ of order $n$, $D_i(G \circ H, x) = [D_i(H, x)]^nI\big(G,\frac{x}{D_i(H, x)}\big)$\!.
\end{corollary}

The Theorem \ref{cd} is useful to build many different graphs with the same independent domination polynomial. If $q$ and $H$ are 
 fixed, then different partitions of $V(G)$ into $q$ cliques gives different graphs with the same independent domination polynomial.

Now, we present various results that   the compound of some special graphs preserves symmetry, unimodality, log-concavity or reality of zeros of independent domination polynomials.  We need the following results:

\begin{lemma}{\rm \cite{Stan}}\label{l31}
	Let $f(x)$ and $g(x)$ be polynomials with positive coefficients.

	{\rm $(i)$} If both $f(x)$ and $g(x)$ are log-concave, then so is their product $f(x)g(x)$.

	{\rm $(ii)$} If $f(x)$ is log-concave, and $g(x)$ is unimodal, then  their product $f(x)g(x)$ is unimodal.

{\rm $(iii)$} If both $f(x)$ and $g(x)$ are symmetric and unimodal, then so is their product $f(x)g(x)$.
\end{lemma}

\medskip 
The following theorem gives a characterization of the graphs having log-concave or real independence domination roots.
\begin{theorem}\label{cdi}
Given two graphs $G$ and $H$, let $\mathcal{C}$ be a clique cover of $G$. Let $D_i(H, x) = ax^2+bx,$ where $a, b$ are nonnegative integers. 

 {\rm $(i)$} If both $I(G, x)$ and $D_i(H, x)$ have only real roots, then so does $D_i(G^{\Delta}(H),x)$.
 
 {\rm $(ii)$} If $I(G, x)$ is log-concave and $a = 0$, then so is $D_i(G^{\Delta}(H),x)$.
\end{theorem}
\begin{proof}
Let $|\mathcal{C}|=q$ and $I(G, x) =\prod\limits_{i=1}^{\alpha(G)}(a_ix + 1)$, where $a_i\geq 0$. Since $I(G, x)$ has only real roots, by 
Theorem \ref{cd} we have 
\begin{eqnarray*}
D_i(G^{\Delta}(H), x) &=&[D_i(H,x)]^q\prod\limits_{i=1}^{\alpha(G)}\big(\frac{a_ix}{ax^2+bx} + 1\big)\\
&=&[D_i(H,x)]^{q-\alpha(G)}\prod\limits_{i=1}^{\alpha(G)}[ax+(b+a_i)].
\end{eqnarray*}
Note that $ax+(b+a_i)$ has only real roots, and $(ii)$ follows from part $(ii)$ of Lemma \ref{l31}. So  $(i)$ and $(ii)$ hold. 
\qed
\end{proof}

\begin{corollary} The independent domination polynomial of $K_{t,n}\circ K_1$ is log-concave for every $t$ and is therefore unimodal.
\end{corollary}
\begin{proof}
To show that $D_i(K_{t,n}\circ K_1)$ is log-concave, we only need to prove that $I(K_{t,n}, x)$
is log-concave by virtue of Theorem \ref{cdi} (ii) for $H = K_1$. We have 
\[I(K_{t,n},x)=(1+x)^t+(1+x)^n-1.\]
Without loss of generality, we can assume $t \leq n$. Thus,
\[(1+x)^t+(1+x)^n=(1+x)^t[1+(1+x)^{n-t}].\]
It follows from Lemma \ref{l31} (ii) that $(1+x)^t+(1+x)^n$ is log-concave. Consequently, it is
clear that $I(K_{t,n}, x)$ is log-concave. This completes the proof. \qed
\end{proof}

\begin{remark}
The graph $H$ in Theorem \ref{cdi} can be a disconnected graph.
\end{remark}


\begin{definition}
 A $(k, n)$-path, denoted by $P^k_n$, begins with $k$-clique on $\{v_1, v_2, \dots, v_k\}$. For $i= k+ 1$ to $n$, let vertex $v_i$ be adjacent to vertices $\{v_{i-1}, v_{i-2}, \dots, v_{i-k}\}$ only. (see Figure \ref{Figure1}).
\end{definition}

\begin{figure}[!ht]
\hspace{3.9cm}
\includegraphics[width=6.cm,height=2.4cm]{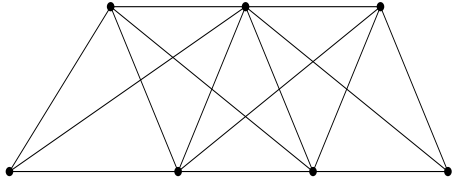}
\caption{\label{Figure1} The  3-path on 7 vertices. }
\end{figure}

Since $P^k_n$ is a claw-free graph, by Theorems \ref{thid} and \ref{cdi} have the following. 
\begin{corollary}
	 Assume that $G$ is  a graph, $D_i(G, x)$ has only real zeros and $D_i(G, x) = ax^2 +bx$, where $0 < a, b \in \mathbb{N}$. Then $D_i((P^k_n)^{\Delta}(G) , x)$ has only real roots. 
\end{corollary}

Noting that the graph $G$ is claw-free, we immediately obtain the following corollary by virtue of
Theorems \ref{thid} and \ref{cdi}.
\begin{corollary}\label{cc} 
Let $H$ be a graph with $\alpha(H) \leq 2$ and $\mathcal{C}$ be a clique cover of another claw-free graph  $G$. If $D_i(H, x)$ has only real zeros, then so does $D_i(G^{\Delta}(H), x)$. In particular, so does $D_i(G\circ H, x)$.
\end{corollary}

\begin{example} 
Consider the centipede graph, $P_n\circ K_1$,  and the caterpillar graph, $P_n\circ \bar{K_2}$: Since $P_n$, i.e., the path with $n$ vertices, is a claw-free graph, by Corollary \ref{cc},  $D_i(P_n\circ K_1,x)$, and $D_i(P_n\circ \bar{K_2},x)$ have only real roots.
\end{example}

\begin{example} 
The $n$-sunlet, $C_n \circ K_1$, where $C_n$ is the cycle with $n$ vertices. By Corollary \ref{cc},  $D_i(C_n\circ K_1, x)$ has only real roots,  since $C_n$ is a claw-free graph. In addition, we also can verify that $D_i(C_n\circ K_r, x)$  have only real zeros for $r\geq 1$.
\end{example}

In \cite{Lev} Levit and Mandrescu constructed a family of graphs $H_n$  from the path $P_n$  by
the ``clique cover construction", as shown in Figure \ref{figure1}, for even $n$,  we take $\mathcal{C} =\{\{1,2\}, \{3,4\}, ...,$ $ \{n-1,n\}\}$, and
for odd $n$, we take  $\mathcal{C} = \{\{1\},\{2,3\}, ..., \{n-3,n-2\},\{n-1,n\}\}$.  By $H_0$ we mean the null graph. 

\begin{figure}[!ht]
\hspace{1.9cm}
\includegraphics[width=11cm,height=4.1cm]{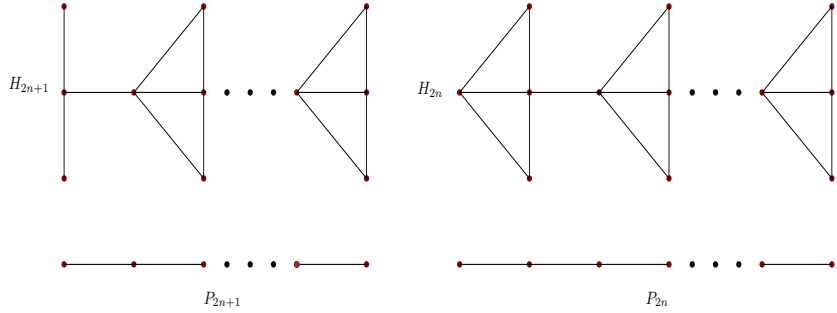}
\caption{ \label{figure1} Graphs $H_{2n+1}$ and $H_{2n}$, respectively. }
\end{figure}

\begin{example} 
Consider the $H_n$ graphs, $P_n^{\Delta}(\bar{K_2})$: Since $P_n$  is a claw-free graph, so by  Theorems \ref{thid} and \ref{cdi},  $D_i(H_n, x)$ has only real zeros. Consequently, $D_i(H_n, x)$ is log-concave and unimodal.
\end{example}

\bigskip
 

\section{Some graphs related to paths }

In this section, we  count the number of independent dominating sets of paths and some graphs related to paths. 

\subsection{Independent domination polynomial of paths} 

 A path is a connected graph in which two vertices have degree
one and the remaining vertices have degree two.  The following result was also proven independently in \cite{Dod} as  Theorem 4.2.

\begin{theorem} {\rm\cite{Dod}}
	 For every $n\geq 4$, 
\[D_i(P_n, x) = xD_i(P_{n-2}, x) + xD_i(P_{n-3}, x),\]
where $
D_i(P_1, x) = x, ~ D_i(P_2, x) = 2x$ and $D_i(P_3, x) = x^2 + x$.
\end{theorem} 
\begin{proof}
If the first vertex of the path is in an independent dominating set, then the second is dominated and
therefore it can not be in the independent dominating set. This case will be counted by $xD_i(P_{n-2}, x)$. If the
first vertex is not in an independent dominating set, then the second vertex must be in the independent dominating set. This gives
$xD_i(P_{n-3}, x)$ and the theorem follows.\qed
\end{proof}

\begin{corollary} \label{corp}
 For every $n\geq 4$,
\[d_i(P_n,k)= d_i(P_{n-2},k-1) + d_i(P_{n-3},k-1)\]
 with initial conditions  $d_i(P_1,1) = 1, d_i(P_2,1) = 2, d_i(P_3,1) = 1$ and $d_i(P_3,2) = 1$.
\end{corollary} 

Moreover, we can prove an explicit formula for the independent domination polynomial
of the path $P_n$. To do this,  we consider  the ordinary generating function for the numbers  $d_i(P_n,k)$ which we denote it simply by $d(n,k)$. 

\begin{theorem} 
	If $F(x, y) =\sum\limits_{n\geq 1}\sum\limits_{k\geq 1} d_i(n,k) x^n y^k$, then \[F(x, y) =\frac{x(1+x)^2y}{1-(x^2+x^3)y}\]
\end{theorem}
\begin{proof}
Consider the following identity: 
\begin{eqnarray*}
\big(1-(x^2+x^3)y\big)F(x, y) &=&F(x, y)-x^2y F(x, y)-x^3yF(x, y)\\
&=&\sum\limits_{n\geq 1}\sum\limits_{k\geq 1} d_i(n,k) x^n y^k - \sum\limits_{n\geq 1}\sum\limits_{k\geq 1} d_i(n,k) x^{n+2} y^{k+1}\\
&& -\sum\limits_{n\geq 1}\sum\limits_{k\geq 1} d_i(n,k) x^{n+3} y^{k+1}\\
&=&\sum\limits_{n\geq 1}\sum\limits_{k\geq 1} d_i(n,k) x^n y^k - \sum\limits_{n\geq 3}\sum\limits_{k\geq 2} d_i(n-2,k-1) x^{n} y^{k}\\
&& -\sum\limits_{n\geq 4}\sum\limits_{k\geq 2} d_i(n-3,k-1) x^{n} y^{k}\\
&=&d_i(1,1) x y + d_i(2,1) x^{2} y + d_i(3,1) x^{3} y + d_i(3,2) x^{3} y^{2} \\
&&+\sum\limits_{n\geq 4}\sum\limits_{k\geq 2} d_i(n,k) x^{n} y^{k}-d_i(1,1) x^{3} y^{2} \\
&&  -\sum\limits_{n\geq 4}\sum\limits_{k\geq 2} d_i(n-2,k-1) x^{n} y^{k} \\
&&-\sum\limits_{n\geq 4}\sum\limits_{k\geq 2} d_i(n-3,k-1) x^{n} y^{k} \\
&=& xy+ 2x^2y+x^3y+x^3y^2-x^3y^2-\sum\limits_{n\geq 4}\\
&&\sum\limits_{k\geq 2}\big(d_i(n,k) - d_i(n-2,k-1) - d_i(n-3,k-1) \big)x^{n} y^{k}\\
&=&x(1+x)^2y
\end{eqnarray*}
where $d_i(n,k)=d_i(P_n,k) = 0$ if $n < k$.\qed
\end{proof}

If we determine the formal power series of $F(x, y)$ expanded in powers of $y$, we have 
\begin{eqnarray*}
F(x, y) &=&\frac{x(1+x)^2y}{1-(x^2+x^3)y}\\
&=&x(1+x)^2y \sum\limits_{k\geq 0} (x^2+x^3)^ky^k\\
&=& \sum\limits_{k\geq 1} x(1+x)^2(x^2+x^3)^{k-1}y^k\\
&=& \sum\limits_{k\geq 1}\big(\sum\limits_{n\geq 1} d_i(P_n,k)x^n\big)y^k,
\end{eqnarray*}
and then for every $k\geq 1$
\begin{eqnarray*}\label{eq}
\sum\limits_{n\geq 1} d_i(P_n,k)x^n &=& x(1+x)^2(x^2+x^3)^{k-1}\\
&=& x^{2k-1}(1+x)^{k+1}
\end{eqnarray*}
is a polynomial denoted by $D_i(P_n,x)$ such that
\[D_i(P_n,x)=\sum\limits_{n\geq k}d_i(P_n,k)x^n = \sum\limits_{t\geq 0}d_i(P_{k+t},k)x^{k+t}\] 

since $D_i(P_n,k) = 0$ if $n < k, k>\alpha(P_n)$ and the well-known fact that  $\gamma_i(P_n) = \lceil \frac{n}{2}\rceil$ for every $n \geq 1$.

\begin{theorem}
For every $k \geq 1$ and $t \geq 0$, the number of independent dominating $k$-sets of the path $P_{k+t}$  is 
\[ d_i(P_{k+t},k)= {k+1 \choose t-k+1}.\]
\end{theorem}
\begin{proof} 
For every $k \geq 1$ and $t \geq 0$, the number of independent dominating $k$-sets of the path $P_{k+t}$  is 
\[ d_i(P_{k+t},k)= {k+1 \choose t-k+1}.\] 
\begin{eqnarray*}
D_i(P_{k+t},x)&=&x^{2k-1}(1+x)^{k+1}\\
&=& x^{2k-1}\sum\limits_{l=0}^{k+1}{k+1\choose l}x^l\\
&=&\sum\limits_{l=0}^{k+1}{k+1\choose l}x^{l+2k-1}\\
&=&\sum\limits_{t=0}^{2k}{k+1 \choose t-k+1}x^{k+t}
\end{eqnarray*}
Therefore, 
\[d_i(P_{k+t},k)= {k+1 \choose t-k+1}.\]
\end{proof}
\begin{corollary}
Let $k$ be a positive integer. The number of the dominating sets of minimum cardinality for paths is
\[
\left\{
\begin{array}{lr}
d_i(P_{3k},k)= 1 & \quad\mbox{ if n=3k,}\\[7pt]
d_i(P_{3k+1},k+1)= {k+2 \choose k}&\quad\mbox{ if n=3k+1,} \\[7pt]
d_i(P_{3k+2},k+1)= {k+2 \choose k+1}&\quad\mbox{ if n=3k+2.}
\end{array}
\right.
\]
\end{corollary}

We have  the following corollary.
\begin{corollary}
For each natural number $n$, the independent domination polynomial of the path $P_n$ is unimodal. 
\end{corollary}
\subsection{Independent dominating sets of some graphs related to path}

A book graph $B_n$, is defined as follows
$V(B_n)=\{u_1,u_2\}\cup \{v_i, w_i : 1\leq i\leq n\}$ and $E(B_n)=\{u_1u_2\}\cup \{u_1v_i,~ u_2w_i,~v_iw_i : 1\leq i\leq n\}$.
We consider  the generalized book graph $B_{n,m}$ with vertex and edge sets by
$V(B_{n,m})=\{u_i: 1\leq i \leq m-2\}\cup \{v_i, w_i : 1\leq i\leq n\}$ and $E(B_{n,m})=\{u_iu_{i+1}:   1\leq i \leq m-3\}\cup \{u_iw_j : 1\leq j\leq n,~i=m-2\}\cup \{ u_1v_i : 1\leq i\leq n\}\cup \{v_iw_i : 1\leq i\leq n\}$ (see Figure \ref{figurebook}). Here, we investigate the independent dominating sets of of the book, and generalized book graphs and count the number of independent dominating sets of these graths.

\begin{figure}[ht]
\hspace{1.cm}
\includegraphics[width=12.cm,height=3.cm]{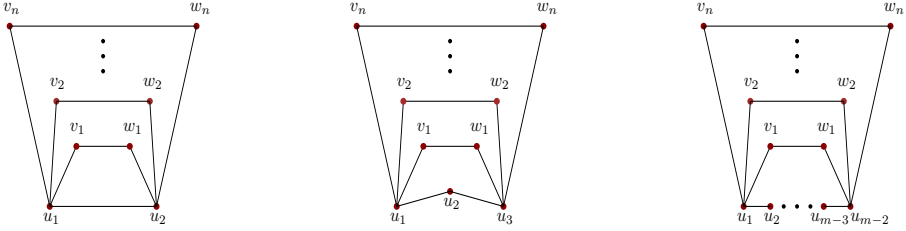}
\caption{ \label{figurebook} Graphs $B_n,~B_{n,5}$ ~and ~ $B_{n,m} $,~ respectively.}
\end{figure}

The following lemma gives a formula for the independent domination polynomial of book graphs.
\begin{lemma}
The independent domination polynomial of the book graph, $B_{n}$, for $n\geq 2$ is given by
\begin{equation*}
D_i(B_{n},x)= (2^n-2)x^n+2x^{n+1}.
\end{equation*}
\end{lemma}
\begin{proof}
Let $B_{n}$ be the book graph. It suffices to show that every 
 independent dominating set of this graph has
size $n$ or $n+1$, and these sets are accounted for exactly once in the above statement. 

The independent dominating sets of each size are one of the following forms:

$(i)$ Choose at least one $v_i$ and at least one $w_i$.  Note that the independent dominating sets of this form  are of size $n$.

$(ii)$ Choose the vertex $u_1~ (u_2)$, and then choose all of the $w_i$ (or all of the $v_i$). Note that the independent dominating sets of this form  are of size $n+1$.

Part $(i)$ accounts for the term $(2x)^n-2x^n$ and part $(ii)$ accounts for the term $2x^{n+1}$. \qed
\end{proof}

We have the following corollary.
\begin{corollary}
For each natural number $n$, $D_i(B_n, x)$ has only real zeros and so  is unimodal.
\end{corollary}

The following theorem gives a formula for the independent domination polynomial of generalized book graphs.
\begin{theorem}\label{bmn}
The independent domination polynomial of the generalized book graph, $B_{n,m}$, for $n\geq 2, ~m\geq 3$ is given by
\begin{equation*}
D_i(B_{n,m},x)= (2^n-2)x^nD_i(P_{m-4},x)+2x^{n+1}D_i(P_{m-5},x) +(x^2+2x^{n+1})D_i(P_{m-6},x),
\end{equation*}
where for all  $j\leq 0,~~D_i(P_{j},x)=1$.
\end{theorem}
\begin{proof}
Let $B_{n,m}$ be the generalized book graph. It suffices to show that every independent dominating set is accounted for exactly once in the above statement.

The independent dominating sets of each size are one of the following forms:

$(i)$ Choose at least one $v_i$ and at least one $w_i$, and then choose any independent dominating set of the induced graph  $\langle u_2, u_3, \cdots, u_{m-3}\rangle$. In this case the independent dominating set contains non of the vertices $u_1$ and $u_{m-2}$. 

$(ii)$ Choose all of the $w_i$ and the vertex $u_2$ (or all of the $v_i$ and the vertex $u_{m-3}$) and then choose any independent dominating set of the  induced graph  $\langle  u_4, \cdots, u_{m-3}\rangle$ 
 $ (\langle u_2, u_3, \cdots, u_{m-5}\rangle)$. In this case the independent dominating set contains non of the vertices $u_1$ and $u_{m-2}$.

$(iii)$ Choose the vertex $u_1~ (u_{m-2})$, then choose all of the $w_i$ (or all of the $v_i$) and then choose any independent dominating set of the  induced graph  $\langle  u_3, \cdots, u_{m-3}\rangle$ 
 $ (\langle u_2, u_3, \cdots, u_{m-4}\rangle)$. In this case the independent dominating set contains one of the vertices $u_1$ or $u_{m-2}$. 

$(iv)$ Choose the vertices $u_1$ and $u_{m-2}$ and then choose any independent dominating set of the induced graph  $\langle u_3, u_4, \cdots, u_{m-4}\rangle$. In this case the independent dominating set contains  the both vertices $u_1$ and $u_{m-2}$. 

Part $(i)$ accounts for the term $(2^n-2)x^nD_i(P_{m-4},x)$, part $(ii)$ accounts for the term $2x^{n+1}D_i(P_{m-6},x)$, part $(iii)$ accounts for the term $2x^{n+1}D_i(P_{m-5},x)$, and part $(iv)$ accounts for the term $x^2D_i(P_{m-6},x)$.\qed
\end{proof}

Since $\gamma_i(P_n) = \lceil \frac{n}{2}\rceil$ for every $n \geq 1$, and by Theorem \ref{bmn} have the following result. 
\begin{corollary}
The independent domination number of the generalized book graph, $B_{n,m}$, for $n\geq 2, ~m\geq 3$  equals to 
\[ \gamma_i(B_{n,m})=\min\{\max\{n,n+\lceil\frac{m-4}{2}\rceil\!\}\!, \max\{n+1,n+1+\lceil\frac{m-5}{2}\rceil\!\}\!, \max\{2,2+\lceil\frac{m-6}{2}\rceil\!\}\}.
\]
\end{corollary}

The friendship (or Dutch-Windmill) graph $F_n$ is a graph that can be constructed by the coalescence of $n$
copies of the cycle graph $C_3$ of length $3$ with a common vertex. The Friendship Theorem of Paul Erd\H{o}s,
Alfred R\'{e}nyi and Vera T. S\'{o}s \cite{erdos}, states that graphs with the property that every two vertices have
exactly one neighbour in common are exactly the friendship graphs. 
Let $n$ and $q\geq 3$ be any positive integer and  $F_{q,n}$ be the {\em generalized friendship graph}  formed by a collection of $n$ cycles (all of order $q$), meeting at a common vertex.  (see Figure \ref{figtetragons}). The generalized friendship graph may also be referred to as a flower \rm\cite{schi}. Now, we study the independent dominating sets of of the Friendship, and generalized Friendship graphs and count the number of independent dominating sets of these graths.

\begin{figure}[ht]
\hspace{2.0cm}
\includegraphics[width=10.5cm,height=2.5cm]{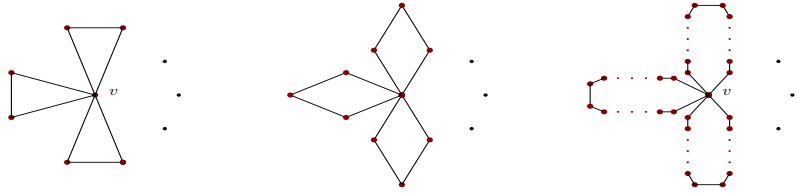}
\caption{\label{figtetragons} The flowers  $F_{n},~ F_{4,n}$ and $F_{q,n}$, respectively. }
\end{figure}

\begin{theorem}
$(a)$ The independent domination polynomial of the friendship graph, $F_{n}$,  for $n\geq 2$ is given by
\begin{equation*}
D_i(F_{n},x)= x+(2x)^n.
\end{equation*}

$(b)$ The independent domination polynomial of the generalized book graph, $F_{q,n}$, for $n\geq 2, ~q\geq 4$ is given by
\begin{equation*}
D_i(F_{q,n},x)= x(D_i(P_{q-3},x))^n+nxD_i(P_{q-3},x) (D_i(P_{q-1},x))^{n-1}.
\end{equation*}
\end{theorem}
\begin{proof}
$(a)$ Consider graph $F_n$ and the vertex $v$ in the common cycles (see Figure \ref{figtetragons}). First assume that the vertex $v$ is contained in the independent dominating set, then the independent dominating sets don't contain $v$.

$(b)$  Consider graph $F_{q,n}$ and the vertex $v$ in the common cycles (see Figure \ref{figtetragons}). 
It suffices to show that every independent dominating set, $S$, is accounted for exactly once in the above statement.

The independent dominating sets of each size are one of the following forms:

$(i)$ If the vertex $v\in S$, then other members in $S$ are chosen of the vertex set of each cycle minus  $N[v]$.

$(ii)$ If the vertex $v\notin S$, so at least one vertex of $N(v)$ in one of cycles, $C_1$, is contained in $S$ (say $x\in S$), then   other members in $S$ are chosen of the vertex set of each cycles else $C_1$ minus $v$ and from   $V(C_1)\smallsetminus N[x]$.

Part $(i)$ accounts for the term $x(D_i(P_{q-3},x))^n$, and part $(ii)$ accounts for the term $nxD_i(P_{q-3},x) (D_i(P_{q-1},x))^{n-1}$.\qed
\end{proof}

Note that for each natural number $n$, $D_i(F_n, x)$ has only real zeros, and so  $D_i(F_n, x)$ is unimodal.


\bigskip

\noindent{\bfseries{Acknowledgements}}
	The authors acknowledge the financial support from  Iran National Science Foundation (INSF),  Tehran and Yazd University research affairs (research project INSF-YAZD 96010014).

\end{document}